\documentclass[12pt]{amsart}

\usepackage{amssymb}
\usepackage{bm}
\usepackage{graphicx}
\usepackage{psfrag}
\usepackage{color}
\usepackage{url}
\usepackage{algpseudocode}
\usepackage{fancyhdr}
\usepackage{xy}
\input xy
\xyoption{all}

\newtheorem{theorem}{Theorem}[section]
\newtheorem{prop}[theorem]{Proposition}
\newtheorem{lemma}[theorem]{Lemma}
\newtheorem{cor}[theorem]{Corollary}
\theoremstyle{definition}
\newtheorem{definition}[theorem]{Definition}
\newtheorem{example}[theorem]{Example}
\numberwithin{equation}{section}

\newcommand\cc{\mathbb{C}}
\newcommand\ff{\mathbb{F}}
\newcommand\nn{\mathbb{N}}
\newcommand\qq{\mathbb{Q}}
\newcommand\rr{\mathbb{R}}
\newcommand\zz{\mathbb{Z}}
\newcommand\pval{\mathsf{v}_p}

\begin{document}

	\mbox{}
	\title{On the Atomic Structure of Puiseux Monoids}
	\author{Felix Gotti}
	\address{Mathematics Department\\UC Berkeley\\Berkeley, CA 94720}
	\email{felixgotti@berkeley.edu}
	\date{\today}
	
	\begin{abstract}
		In this paper, we study the atomic structure of the family of Puiseux monoids, i.e, the additive submonoids of $\qq_{\ge 0}$. Puiseux monoids are a natural generalization of numerical semigroups, which have been actively studied since mid-nineteenth century. Unlike numerical semigroups, the family of Puiseux monoids contains non-finitely generated representatives. Even more interesting is that there are many Puiseux monoids which are not even atomic. We delve into these situations, describing, in particular, a vast collection of commutative cancellative monoids containing no atoms. On the other hand, we find several characterization criteria which force Puiseux monoids to be atomic. Finally, we classify the atomic subfamily of strongly bounded Puiseux monoids over a finite set of primes.
	\end{abstract}
	
	\maketitle
	
	\section{Introduction} \label{sec:intro}
	
	A \emph{Puiseux monoid} is an additive submonoid of the non-negative rational numbers. The family of Puiseux monoids is a natural generalization of that one comprising all numerical semigroups. In this paper, we explore the atomic structure of the former family, which is far more complex than the atomic structure of numerical semigroups. However, the controlled atomic behavior of numerical semigroups will guide our initial approach to Puiseux monoids.
	
	Numerical semigroups are atomic monoids that have been systematically studied since the mid-nineteen century; see the monograph \cite{GR09} of Garc\'ia-S\'anchez and Rosales. In algebraic geometry, Noetherian local domains whose integral closures are finitely generated modules and discrete valuation rings pop up very often, and their associated valuations turn out to be numerical semigroups. Many properties of the previously mentioned domains can be fully characterized in terms of their valuation numerical semigroups. For more details, see \cite{BDF97}.
	
	Understanding the atomicity of Puiseux monoids can set the groundwork for a future exploration of the arithmetic properties and factorization invariants of their atomic subfamilies. Once we obtain good insight of the algebraic properties of Puiseux monoids, we might expect to use this family of commutative monoids to understand certain behaviors of Puiseux domains (see, e.g., \cite[Sec.~13.3]{dE95}) and other subdomains of power series with rational exponents. This would mirror the way numerical semigroups have been used to understand many attractive properties of subdomains of power series with natural exponents (see \cite{BDF97}).
	
	Although significant effort has been put in exploring the arithmetic invariants of many families of atomic monoids (see, for instance, \cite{CCMMP14,CGP14,CS10,wS05}), very little work has gone into an attempt to classify them. In this paper, we find an entirely new family of atomic monoids hidden inside the realm of Puiseux monoids, increasing the current spectrum of atomic monoids up to isomorphism and, therefore, contributing to a classification of the aforementioned family. In addition, Puiseux monoids provide a source of examples of both atomic and non-atomic monoids. This new arsenal of examples might help to test several existence conjectures concerning commutative semigroups and factorization theory.
	
	The family of Puiseux monoids contains a vast collection of non-atomic representatives, monoids containing non-unit elements with no factorizations into irreducibles. Even more surprising, Theorem \ref{theo:sufficiency for anatomicity} identifies a subfamily of antimatter representatives, Puiseux monoids possessing no irreducible elements. In contrast, there are various subfamilies of Puiseux monoids whose members are atomic even when they are not isomorphic to numerical semigroups. We devote this paper to introduce and study the fascinating atomic structure of Puiseux monoids.
	
	In Section \ref{sec:preliminary}, we establish the terminology we will be using throughout this paper. In Section \ref{sec:Atomic Characterization of Puiseux Monoids}, after pointing out how Puiseux monoids naturally appear in commutative ring theory, we introduce some members of the targeted family, illustrating how much Puiseux monoids differ from numerical semigroups in terms of atomic configuration. Once we have highlighted the wildness of the atomic structure of the family being investigated, we show that its atomic members are precisely those containing a minimal set of generators (Theorem \ref{theo:finitely generated rational monoids}). We then present two sufficient conditions for atomicity (Proposition \ref{prop:bounded below valuation implies atomicity} and Theorem \ref{theo:sufficient condition for atomicity}). In Section \ref{sec:Strongly Bounded PM}, we introduce the subfamily of \emph{strongly bounded} Puiseux monoids, presenting simultaneously atomic and antimatter subfamilies failing to be strongly bounded. In the last section, we study the atomic configuration of strongly bounded Puiseux monoids. We present a sufficient condition for strongly bounded Puiseux monoids to be antimatter (Theorem \ref{theo:sufficiency for anatomicity}). To conclude, we dedicate the second part of the last section to the classification of the atomic subfamily of strongly bounded Puiseux monoids \emph{over} a finite set of primes (defined in Section \ref{sec:Atomic Characterization of Puiseux Monoids}). \\
	
	\section{Preliminary}\label{sec:preliminary}
	
	We begin by presenting some of the terminology related to the atomicity of commutative cancellative monoids. Then we briefly mention a few basic properties of numerical semigroups, the objects we generalize in this work. Our goal in this section is not to formally introduce the elementary concepts and results of commutative semigroups and factorization theory, but rather to fix notation and establish the nomenclature we will use later. For extensive background information on commutative semigroups and non-unique factorization theory, we refer readers to the monographs \cite{pG01} of Grillet and \cite{GH06} of Geroldinger and Halter-Koch, respectively.
	
	We use the double-struck symbols $\mathbb{N}$ and $\mathbb{N}_0$ to denote the sets of positive integers and non-negative integers, respectively. Moreover, if $r$ is a real number, we will denote the set $\{z \in \zz \mid z \ge r\}$ simply by $\zz_{\ge r}$; with a similar intention, we will use the notations $\zz_{> r}$, $\qq_{\ge r}$, and $\qq_{> 0}$. If $S \subseteq \qq$, we often write $S^\bullet$ instead of $S \setminus \{0\}$. For $r \in \qq_{> 0}$, we denote the unique $a,b \in \nn$ such that $r = a/b$ and $\gcd(a,b)=1$ by $\mathsf{n}(r)$ and $\mathsf{d}(r)$, respectively. If $R \subseteq \qq_{>0}$, we call the sets $\mathsf{n}(R) = \{\mathsf{n}(r) \mid r \in R\}$ and $\mathsf{d}(R) = \{\mathsf{d}(r) \mid r \in R\}$ the \emph{numerator} and \emph{denominator} \emph{set} of $R$, respectively.
	
	Unless otherwise specified,  the word \emph{monoid} in this paper means commutative cancellative monoid. Let $M$ be a monoid. Because every monoid is assumed to be commutative, unless we state otherwise, we will always use additive notation; in particular, $``+"$ denotes the operation of $M$, while $0$ denotes the identity element. The invertible elements of a monoid are called \emph{units}, and the set of all units of $M$ is denoted by $M^\times$. The monoid $M$ is said to be \emph{reduced} if $M^\times = \{0\}$. If $M$ is generated by a subset $S$, we write $M = \langle S \rangle$. The monoid $M$ is \emph{finitely generated} if $M = \langle S \rangle$ for some finite set $S$. For a brief but precise exposition of finitely generated commutative monoids, readers might find \cite{GR99} very useful. An element $a \in M \! \setminus \! M^\times$ is \emph{irreducible} or an \emph{atom} if $a = x + y$ implies either $x$ or $y$ is a unit. The set of atoms of $M$ is denoted by $\mathcal{A}(M)$. The monoid $M$ is \emph{atomic} if every non-unit element of $M$ can be expressed as a sum of atoms, i.e., $M = \langle \mathcal{A}(M) \rangle$.
	
	We briefly comment on general properties of numerical semigroups. A \emph{numerical semigroup} $N$ is a submonoid of the additive monoid $\nn_0$ such that $\nn_0 \setminus \! N$ is finite. Every numerical semigroup has a unique minimal set of generators, which happens to be finite. For $n \in \nn$, if $N = \langle a_1, \dots, a_n \rangle$ is minimally generated by $a_1, \dots, a_n \in \nn$, then $\gcd(a_1, \dots, a_n) = 1$ and $\mathcal{A}(N) = \{a_1, \dots, a_n\}$. Consequently, every numerical semigroup is atomic and has finitely many atoms. The family of numerical semigroups has been intensely studied for more than three decades. For an entry point to the realm of numerical semigroups, readers might consider \cite{GR09} to be a valuable resource.
	
	Let $N = \langle a_1, \dots, a_n \rangle$ be a minimally generated numerical semigroup. The \emph{Frobenius number} of $N$, denoted by $F(N)$, is the greatest natural number not contained in $N$, i.e., the smallest integer $F(N)$ such that for all $b \in \nn$ with $b > F(N)$ the Diophantine equation $a_1x_1 + \dots + a_nx_n = b$ has a solution in $\nn_0^n$. We will need later the following result (taken from \cite{aB42}), which gives an upper bound for the Frobenius number $F(N)$ in terms of the minimal set of generators $a_1, \dots, a_n$.
	
	\begin{theorem} \label{theo:Frobenius upper bound}
		Let $N = \langle a_1, \dots, a_n \rangle$ be a minimally generated numerical semigroup, where $a_1 < \dots < a_n$. Then
		\[
			F(N) <  (a_1 - 1)(a_n - 1).
		\]
	\end{theorem}
	
	Numerical semigroups not only have a canonical set of generators, but also exhibit a very convenient atomic structure. In the next section, we explore how these desirable properties behave in the more general setting of Puiseux monoids. \\
	
	\section{Atomic Characterization of Puiseux Monoids} \label{sec:Atomic Characterization of Puiseux Monoids}
	
	We start this section with a brief discussion on how Puiseux monoids show up naturally in commutative ring theory. We explain their connection to the field of Puiseux series, justifying our choice of the name Puiseux. Then we move to study the atomicity of the family of Puiseux monoids; we find an atomic characterization and two sufficient conditions for atomicity.
	
	Let $\ff$ be a field. A \emph{valuation} on $\ff$ is a map $\text{val} \colon \ff \to \rr \cup \{\infty\}$ satisfying the following three axioms:
	\begin{enumerate}
		\setlength\itemsep{3pt}
		\item $\text{val}(r) = \infty$ \ if and only if $r=0$;
		\item $\text{val}(rs) = \text{val}(r) + \text{val}(s)$ \ for all \ $r,s \in \ff^\times$;
		\item $\text{val}(r + s) \ge \min\{\text{val}(r), \text{val}(s) \}$ \ for all \ $r,s \in \ff^\times$.
	\end{enumerate}
	
	\begin{example}
		Take $\ff$ to be the field of Laurent series $\cc((T))$ in the formal variable $T$. Consider the map $\text{val}_L \colon \cc((T)) \to \rr \cup \{\infty\}$ defined by
		\[
			\text{val}_L\bigg(\sum_{n \ge N} c_n T^n\bigg) = \min\{n \in \zz_{\ge N} \mid c_n  \neq 0 \}
		\]
		if $\sum_{n \ge N} c_n T^n \neq 0$, and $\text{val}(0) = \infty$. It is not difficult to verify that the function $\text{val}_L$ is a valuation on $\cc((T))$; this is a standard result that is explained in many introductory textbook in algebra.
	\end{example}
	
	The algebraic closure of the field of Laurent series $\cc((T))$ in the formal variable $T$ is denoted by $\cc\{\!\{T\}\!\}$ and called the field of \emph{Puiseux series}; it was first studied by Puiseux in \cite{vP1850}. We can write the field of Puiseux series as
	\[
		\cc\{\!\{T\}\!\} = \bigcup_{d \in \nn} \cc\big(\big(T^{\frac{1}{d}}\big)\big),
	\]
	where $\cc((T^{1/d}))$ is the field of Laurent series in the formal variable $T^{1/d}$. The nonzero elements in $\cc\{\!\{T\}\!\}$ are formal power series of the form
	\[
		c(T) = c_1T^{\frac{n_1}d} + c_2T^{\frac{n_2}d} + \dots,
	\]
	where $c_1, c_2, \dots$ are complex numbers such that $c_1 \neq 0$, the denominator $d$ is a natural number, and $n_1 < n_2 < \dots$ are integers. Also, the function $\text{val}_P \colon \cc\{\!\{T\}\!\} \to \rr \cup \{\infty\}$ mapping $c(T)$ to $n_1/d$ and $0$ to $\infty$ is a valuation on $\cc\{\!\{T\}\!\}$.
	
	Let $R$ be a subring of $\cc\{\!\{T\}\!\}$. By the second axiom in the definition of valuation, the image of $R$ under $\text{val}_P$ is closed under addition. Since $\text{val}_P(1) = 0$, it follows that $\text{val}_P(R)$ is an additive submonoid of $\qq$. In particular, Puiseux monoids arise naturally in commutative ring theory as images under $\text{val}_P$ of subdomains of $\cc\{\!\{T\}\!\}$ of positive valuations. Given this connection, the monoids investigated in this paper are named Puiseux, honoring the French mathematician Victor A. Puiseux (1820-1883). \\
	
	We now proceed to study the atomic structure of Puiseux monoids. The family of Puiseux monoids is a natural generalization of that of numerical semigroups. The following proposition, whose proof follows immediately, characterizes those Puiseux monoids isomorphic to numerical semigroups.
	
	\begin{prop} \label{prop:finitely generated positive rational monoids}
		A Puiseux monoid is isomorphic to a numerical semigroup if and only if it is finitely generated.
	\end{prop}
	
	Let $M$ be a Puiseux monoid. If $S = \{s_1, s_2, \dots \}$ is a set of rational numbers generating $M$, instead of $M = \langle S \rangle$ sometimes we write $M = \langle s_1, s_2, \dots \rangle$, omitting the brackets in the description of $S$. We say that $M$ is \emph{minimally} generated by $S$ if no proper subset of $S$ generates $M$. In clear contrast to numerical semigroups, there are Puiseux monoids that are neither atomic nor finitely generated; see examples below. In addition, unlike numerical semigroups, not every Puiseux monoid is atomic. In fact, there are nontrivial Puiseux monoids containing no atoms at all. On the other hand, there are atomic and non-atomic Puiseux monoids having infinitely many atoms. The following examples illustrate the facts just mentioned.
	
	\begin{example} \label{ex:nonfinite generated with finite number of atoms trivial}
		Fix a prime number $p$. Let $M$ be the Puiseux monoid generated by the set $S = \{1/p^n \mid n \in \nn\}$. Although $M$ is not finitely generated, its set of atoms is empty. This is because $\mathcal{A}(M) \subseteq S$ and $1/p^n$ is the sum of $p$ copies of $1/p^{n+1}$ for every positive integer $n$.
	\end{example}
	
	\begin{example} \label{ex:atomic example with infinitely many atoms}
		Let $P$ be the set comprising all prime numbers, and consider the Puiseux monoid $M = \langle 1/p \mid p \in P \rangle$. We shall check that $1/p$ is an atom for every $p \in P$. For a prime $p$, suppose
		\[
			\frac{1}{p} = \frac{1}{p_1} + \dots + \frac{1}{p_n},
		\]
		where $n$ is a natural number and the $p_k$ are not necessarily distinct primes. Setting $m = p_1 \dots p_n$ and $m_k = m/p_k$ for $k = 1,\dots,n$, we obtain $m/p = m_1 + \dots + m_n \in \nn$. Therefore $p$ divides $m$, and so $p_k = p$ for some $k$. Thus, $n=1$, which means that $1/p \in \mathcal{A}(M)$. Since $M$ is generated by atoms, it is atomic. Finally, it follows that $M$, albeit atomic, is not isomorphic to a numerical semigroup; to confirm this, note that $M$ contains infinitely many atoms.
	\end{example}
	
	\begin{example} \label{ex:nonfinite generated with infinitely many atoms}
		Let $p_1, p_2, \dots$ be an enumeration of the odd prime numbers. Let $M$ be the Puiseux monoid generated by the set $S \cup T$, where $S = \{1/2^n \mid n \in \nn\}$ and $T = \{1/p_n \mid n \in \nn\}$. It is easy to check, in the same way we did in Example \ref{ex:atomic example with infinitely many atoms}, that $1/p_n$ is an atom of $M$ for each $n \in \nn$. On the other hand, it follows immediately that $1/2^n \notin \mathcal{A}(M)$ for any $n \in \nn$. Since $T \subseteq \mathcal{A}(M) \subseteq S \cup T$ and $S \cap \mathcal{A}(M)$ is empty, one has $\mathcal{A}(M) = T$. We verify now that $1/2^n$ cannot be written as a sum of atoms for any $n \in \nn$. Suppose, by way of contradiction, that for some $n \in \nn$ there exist a positive integer $k$ and non-negative coefficients $c_1, \dots, c_k$ satisfying
		\begin{equation}
			\frac{1}{2^n} = \sum_{i=1}^k c_i\frac{1}{p_i}. \label{eq:example of nonatomic PM with infinitely many atoms}
		\end{equation}
		Multiplying \eqref{eq:example of nonatomic PM with infinitely many atoms} by $m = p_1 \dots p_k$, one gets
		\[
			\frac{m}{2^n} = \sum_{i=1}^k c_i m_i \in \nn,
		\]
		where $m_i = m/p_i$. This implies that $2^n$ divides $m$. Since $m$ is odd, we get a contradiction. Hence $1/2^n$ cannot be written as a sum of atoms for any $n \in \nn$. Consequently, $M$ is a non-atomic monoid with infinitely many atoms.
	\end{example}
	
	Like numerical semigroups, Puiseux monoids are reduced. Therefore the set of atoms of a Puiseux monoid is contained in every set of generators. As we mentioned before, a numerical semigroup has a unique minimal set of generators, namely its set of atoms. Theorem~\ref{theo:finitely generated rational monoids} shows that having a (unique) minimal set of generators characterizes the family of atomic Puiseux monoids.
	
	\begin{theorem} \label{theo:finitely generated rational monoids}
		If $M$ is a Puiseux monoid, the following conditions are equivalent:
		\vspace{-3pt}
		\begin{enumerate}
			\setlength\itemsep{3pt}
			\item $M$ contains a minimal set of generators;
			\item $M$ contains a unique minimal set of generators;
			\item $M$ is atomic.
		\end{enumerate}
	\end{theorem}
	
	\begin{proof}
		First, we show that conditions (1) and (2) are equivalent. Since (1) follows immediately from (2), it suffices to prove (1) implies (2). Suppose $S$ and $S'$ are two minimal sets of generators of $M$. Take an arbitrary $s \in S$. The fact that $M = \langle S' \rangle$ leads to the existence of $n \in \nn$ and $s'_1, \dots, s'_n \in S'$ such that $s = s'_1 + \dots + s'_n$. Because $S$ also generates $M$, for each $i = 1, \dots, n$, we have $s'_i = s_{i1} + \dots + s_{in_i}$ for some $n_i \in \nn$ and $s_{ij} \in S$ for $j \in \{1, \dots, n_i\}$. As a result,
		\[
			s = \sum_{i=1}^n s'_i = \sum_{i=1}^n \sum_{j=1}^{n_i} s_{ij}.
		\]
		The minimality of $S$ implies $n=1$ and, therefore, one gets $s = s'_1 \in S'$. Then $S \subseteq S'$ and, using a similar argument, we can check that $S' \subseteq S$. Hence, if a minimal set of generators exists, then it must be unique.
		
		Now we prove that (1) and (3) are equivalent. First, assume condition (1) holds. Let $S$ be a minimal set of generators of $M$. Let us show that every element in $S$ is an atom. Suppose, by way of contradiction, that $a \in S$ is not an atom. So $a = x + y$ for some $x,y \in M^\bullet$. Since $x$ and $y$ are both strictly less than $a$, each of them can be written as a sum of elements in $S \! \setminus \! \{a\}$. As a result, $a = x + y \in \langle S \! \setminus \! \{a\} \rangle$, contradicting the minimality of $S$. Therefore $S \subseteq \mathcal{A}(M)$. As $\mathcal{A}(M)$ must be contained in any set of generators, $S = \mathcal{A}(M)$. Thus, $M$ is atomic, which is condition~(3). Finally, we check (3) implies (1).  Assume $M$ is atomic, i.e., $M = \langle \mathcal{A}(M) \rangle$. As $M$ is reduced, no atom can be written as a sum of positive elements of $M$. Hence $\mathcal{A}(M)$ is a minimal set of generators.
	\end{proof}
	
	In contrast with numerical semigroups, there are Puiseux monoids containing no minimal sets of generators. When $M$ is not atomic we still have $\mathcal{A}(M) \subseteq S$ for every minimal set of generators $S$. Nevertheless, $\mathcal{A}(M)$ might not generate $M$, as Example~\ref{ex:nonfinite generated with infinitely many atoms} shows. In fact, $M$ can fail to be finitely generated and still have finitely many atoms; Example~\ref{ex:nonfinite generated with finite number of atoms trivial} sheds light upon this situation. \\

	Let $p$ be a prime. For a nonzero integer $a$, define $\pval(a)$ to be the exponent of the maximal power of $p$ dividing $a$, and set $\pval(0) = \infty$. In addition, for $b \in \zz \! \setminus \! \{0\}$, set $\pval(a/b) = \pval(a) - \pval(b)$. It follows immediately that the map $\pval \colon \qq \to \rr \cup \{\infty\}$, which is called the $p$-\emph{adic valuation}, is an actual valuation on $\qq$. In particular, it satisfies the third condition in the definition of valuation given before, that is
	\begin{equation}
		\pval(r + s) \ge \min\{\pval(r), \pval(s) \} \text{ for all } \ r,s \in \qq^\times. \label{eq:semiadditivity of valuations}
	\end{equation}
	
	\begin{definition}
		Let $P$ be a set of primes. A Puiseux monoid $M$ \emph{over} $P$ is a Puiseux monoid such that $\pval(m) \ge 0$ for every $m \in M$ and $p \notin P$.
	\end{definition}
	If $P$ is finite, then we say that $M$ is a \emph{finite} Puiseux monoid over $P$. The Puiseux monoid $M$ is said to be \emph{finite} if there exists a finite set of primes $P$ such that $M$ is finite over $P$. To simplify the notation, if $P$ contains only one prime $p$, we write Puiseux monoid over $p$ instead of Puiseux monoid over $\{p\}$. We will see that the $p$-adic valuation maps play an important role in describing the atomic configuration of Puiseux monoids over $P$. For example, in Proposition \ref{prop:bounded below valuation implies atomicity}, we check that a finite Puiseux monoid over $P$ is atomic if for each $p \in P$ the sequence of $p$-adic valuations of its generators is bounded from below.
	
	The remainder of this section is devoted to finding characterization criteria for atomicity of Puiseux monoids. Proposition \ref{prop:bounded below valuation implies atomicity} and Theorem \ref{theo:sufficient condition for atomicity} identify two atomic subfamilies of Puiseux monoids.
	
	\begin{prop} \label{prop:bounded below valuation implies atomicity}
		For a Puiseux monoid $M$ the following conditions are equivalent:
		\begin{enumerate}
			\item $M$ is finite and $\{\pval(M)\}$ is bounded from below for every prime $p$;
			\item $M$ is finite, and $M = \langle R \rangle$ implies that $\{\pval(R)\}$ is bounded from below for every prime $p$;
			\item The denominator set $\mathsf{d}(M^\bullet)$ is bounded;
			\item If $M = \langle R \rangle$, then the denominator set $\mathsf{d}(R^\bullet)$ is bounded.
		\end{enumerate}
		If one (and so all) of the conditions above holds, then $M$ is atomic. 
	\end{prop}
	
	\begin{proof}
		Condition (1) trivially implies condition (2). Assume condition (2), and let $P$ be a finite set of primes over which $M$ is finite. For all $p \in P$ one has $\{\pval(R)\}$ is bounded from below; therefore $\mathsf{d}(R^\bullet)$ is finite. Taking $m$ to be the product of all elements in $\mathsf{d}(R^\bullet)$, one has $d \mid m$ for all $d \in \mathsf{d}(M^\bullet)$. Hence $\mathsf{d}(M^\bullet)$ is finite and (3) holds. Condition (3) implies condition (4) trivially. Finally, assume condition (4). Since $\mathsf{d}(R^\bullet)$ is bounded, so is $\mathsf{d}(M^\bullet)$. As a consequence, only finitely many primes divide elements in $\mathsf{d}(M^\bullet)$. The boundedness of $\mathsf{d}(M^\bullet)$ also implies that $\{\pval(M)\}$ is bounded from below for every prime $p$, which is condition (1).
		
		Now we will check that condition (2) implies that $M$ is atomic. Suppose that $M$ is finite over $P = \{p_1, \dots, p_n\}$ for some $n \in \nn$. Let $R$ be a subset of rationals such that $M = \langle R \rangle$. Set
		\[
			m_i = \min \big\{0, \min_{r \in R^\bullet}\{{\pval}_i(r)\}\big\}
		\]
		for each $i \in \{1, \dots, n\}$, and take $m = p_1^{-m_1} \dots \, p_n^{-m_n}$. The function $\varphi \colon M \to mM$ defined by $\varphi(x) = mx$ is an isomorphism. An arbitrary $x \in M^\bullet$ can be written as $x = c_1r_1 + \dots + c_kr_k$, where $k \in \nn$ while $c_i \in \nn$ and $r_i \in R^\bullet$ for every $i \in \{1,\dots,k\}$. By inequality \eqref{eq:semiadditivity of valuations}, one has
		\[
			{\pval}_j(mx) = {\pval}_j\bigg(m\sum_{i=1}^k c_i r_i\bigg) \ge \ \min_{1 \le i \le k}\{{\pval}_j(m c_i r_i) \} \ge \ \min_{1 \le i \le k}\{{\pval}_j(c_i)\} \ge 0
		\]
		for $j = 1, \dots, n$. Since $\pval(mM^\bullet) \subseteq \nn_0$ for each $p \in P$, it follows that $mM$ is isomorphic to a numerical semigroup, and so it is atomic. Hence $M$ is also atomic, as expected.
	\end{proof}
	
	If one of the four conditions of Proposition~\ref{prop:bounded below valuation implies atomicity} fails, namely that $\{\pval(r) \mid r \in R\}$ is not bounded from below for some $p \in P$, then $M$ might not be atomic. This is illustrated in Example \ref{ex:nonfinite generated with finite number of atoms trivial}. Besides, if we allowed $|P| = \infty$, Proposition \ref{prop:bounded below valuation implies atomicity} would not hold, as we can see in the next example.
	
	\begin{example}
		Let $P = \{p_1, p_2, \dots\}$ be an infinite set of primes. Then we define the Puiseux monoid over $P$
		\[
			M = \bigg\langle \frac{1}{d_1}, \frac{1}{d_2}, \dots \bigg\rangle, \ \text{ where } \ d_n = p_1 \dots p_n
		\] 
		for every $n \in \nn$. Observe that $1/d_n \notin \mathcal{A}(M)$ for any $n \in \nn$; this is because $1/d_n$ is the sum of $p_{n+1}$ copies of $1/d_{n+1}$. Since $\mathcal{A}(M)$ is contained in any set of generators, the fact that $1/d_n \notin \mathcal{A}(M)$ for every natural $n$ implies that $\mathcal{A}(M)$ is empty. Therefore $M$ is not atomic. In Section \ref{sec:Strongly Bounded PM} we will give a special name to the Puiseux monoids whose set of atoms is empty.
	\end{example}
	
	\begin{theorem} \label{theo:sufficient condition for atomicity}
		Let $M$ be a Puiseux monoid. If $0$ is not a limit point of $M$, then $M$ is atomic.
	\end{theorem}
	
	\begin{proof}
		Assume, by way of contradiction, that $M$ is not atomic. Let $N$ be the nonempty subset of $M^\bullet$ comprising all the elements that cannot be written as a sum of atoms. For $r_1 \in N$ there exist positive elements $r_2$ and $r'_2$ in $M$ such that $r_1 = r_2 + r'_2$. Either $r_2$ or $r'_2$ must be contained in $N$; otherwise $r_1$ would not belong to $N$. Let us suppose then, without loss of generality, that $r_2 \in N$. We have $r_1 > r_2$ and $r_1 - r_2 = r'_2 \in M$. Because $r_2 \in N$, there exist $r_3, r'_3 \in M^\bullet$ such that $r_2 = r_3 + r'_3$ and either $r_3 \in N$ or $r'_3 \in N$. Assume $r_3 \in N$. Again, one obtains $r_2 > r_3$ and $r_2 - r_3 = r'_3 \in M$. Continuing in this fashion, we can build two sequences $\{r_n\}$ and $\{r'_n\}$ of elements of $M$ such that $\{r_n\}$ is decreasing and $r_n - r_{n+1} = r'_{n+1} \in M$ for every $n \in \nn$. Since $\{r_n\}$ is a decreasing sequence of positive terms, it converges, and so it is a Cauchy sequence. This implies that the sequence $\{r'_n\}$ converges to zero. But it contradicts the fact that $0$ is not a limit point of $M$. Thus, $M$ is atomic, which establishes the theorem.
	\end{proof}
	
	The converse of Theorem \ref{theo:sufficient condition for atomicity} does not hold. The next example not only illustrates the failure of its converse, but also indicates that Proposition \ref{prop:bounded below valuation implies atomicity} and Theorem~\ref{theo:sufficient condition for atomicity} are not enough to fully characterize the family of atomic Puiseux monoids.
	
	\begin{example}
		Let $p_1, p_2, \dots$ be an enumeration of the odd prime numbers. Define the sequence of positive integers $\{k_n\}$ as follows. Take $k_1 \in \nn$ to be arbitrary, and once $k_n$ has been chosen, take $k_{n+1} \in \nn$ such that both inequalities $k_{n+1} > k_n$ and $2^{k_{n+1}}p_1\dots p_n > 2^{k_n + 1}p_1\dots p_{n+1}$ hold. Now define the Puiseux monoid
		\[
			M = \langle r_1, r_2, \dots \rangle, \ \text{ where } \ r_n = \frac{p_1 \dots p_n}{2^{k_n}}. 
		\]
		We verify that $0$ is a limit point of $M$ and that $M$ is atomic. The way we defined the sequence $\{k_n\}$ ensures that $r_{n+1} < r_n/2$ for every $n \in \nn$. As a result, the sequence $\{r_n\}$ converges to $0$ and, hence, $0$ is a limit point of $M$. Additionally, for $j \in \nn$, suppose
		\begin{equation}
			r_j = \sum_{i=1}^m c_ir_i = \sum_{i=1}^m c_i \frac{p_1 \dots p_i}{2^{k_i}}, \label{eq:atomicity in zero limit point example}
		\end{equation}
		for some $m \in \nn$ and coefficients $c_1, \dots, c_m \in \nn_0$. As $r_i > r_j$ when $i < j$, we get $c_i = 0$ for $i < j$ and $c_j \in \{0,1\}$. If $c_j = 0$, then \eqref{eq:atomicity in zero limit point example} can be written as
		\begin{equation}
			2^{k_m - k_j}p_1 \dots p_j = \sum_{i=j+1}^m 2^{k_m - k_i}c_ip_1 \dots p_i. \label{eq:atomicity in zero limit point example 2}
		\end{equation}
		Every summand in the right-hand side of \eqref{eq:atomicity in zero limit point example 2} is divisible by $p_{j+1}$, which contradicts that the left-hand side of \eqref{eq:atomicity in zero limit point example 2} is not divisible by $p_{j+1}$. Therefore $c_j = 1$, and so $r_j$ is an atom. Since $M$ is generated by atoms, it is atomic.
	\end{example}

	Clearly, if $M$ is a Puiseux monoid as in Theorem~\ref{theo:sufficient condition for atomicity} (i.e., $0$ is not a limit point of $M$), then every submonoid of $M$ must also be atomic as a result of Theorem~\ref{theo:sufficient condition for atomicity}. However, in general, it is not true that every submonoid of an atomic monoid is atomic. The next example illustrates this observation.
	
	\begin{example}
		Let $\{p_n\}$ be the sequence comprising the odd prime numbers in strictly increasing order. Then consider the Puiseux monoid $M = \langle S \rangle$, where
		\[
			 S = \bigg\{ \frac 1{2^np_n} \ \bigg{|} \ n \in \nn \bigg\}.
		\]
		Since each odd prime divides exactly one element of the set $\mathsf{d}(S)$, it follows that $\mathcal{A}(M) = S$. Hence $M$ is atomic. On the other hand, the element $1/2^n$ is the sum of $p_n$ copies of the atom $1/(2^np_n)$ for every $n \in \nn$. Thus, the monoid
		\[
			N = \langle 1/2^n \mid n \in \nn \rangle
		\]
		contains no atoms, which immediately implies that $N$ is not atomic. Therefore $N$ is a submonoid of the atomic monoid $M$ that fails to be atomic.
	\end{example}
	
	Proposition \ref{prop:bounded below valuation implies atomicity} and Theorem \ref{theo:sufficient condition for atomicity} give rise to large families of atomic monoids. Theorem \ref{theo:sufficient condition for atomicity} applies, in particular, when a Puiseux monoid is generated by an eventually increasing sequence. Many factorization invariants of numerical semigroups have been investigated during the last two decades. For example, the set of lengths, elasticity, delta set, and catenary/tame degree have been actively studied in terms of minimal sets of generators (see \cite{ACHP07,BOP16,CCMMP14,CHK09,mO12} and references therein). Studying these factorization invariants on the atomic Puiseux monoids provided by Theorem \ref{theo:sufficient condition for atomicity} would contribute significantly to understanding their algebraic and combinatorial structure. \\
		
	\section{Strongly Bounded Puiseux Monoids} \label{sec:Strongly Bounded PM}
	
	In this section, we restrict attention to the atomic structure of those Puiseux monoids that can be generated by a subset $S \subset \qq$ satisfying that $\mathsf{n}(S)$ is bounded. To be more precise, we say that a subset $S$ of rational numbers is \emph{strongly bounded} if its numerator set, $\mathsf{n}(S)$, is bounded.
	
	\begin{definition}
		A Puiseux monoid is \emph{bounded} (resp., \emph{strongly bounded}) if it can be generated by a bounded (resp., strongly bounded) set of rational numbers.
	\end{definition}
	
	Although, for the remainder of this paper, we focus on studying the subfamily of strongly bounded Puiseux monoids, they are by no means the only subfamily containing atomic representatives. The following proposition explains a way of constructing a family of atomic Puiseux monoids whose members are not strongly bounded.
	
	\begin{prop} \label{prop:PM atomic not strongly bounded}
		There exist infinitely many atomic Puiseux monoids that are not strongly bounded.
	\end{prop}
	
	\begin{proof} 
		Let $p$ be a prime. Suppose $\{a_n\}$ is the sequence recurrently defined as follows. Choose $a_1 \in \nn$ such that $a_1 > p$. Suppose $a_1, \dots, a_n$ have been selected. Take $a_{n+1} \in \nn$ so that $\gcd(a_{n+1},p) = 1$ and $a_{n+1}/p^{n+1} > a_n/p^n$. Consider the Puiseux monoid over $p$
		\[
			M = \langle S \rangle, \ \text{ where } \ S = \bigg\{ \frac{a_n}{p^n} \ \bigg{|} \ n \in \nn \bigg\}.
		\]
		Let us check that $a_n/p^n \in \mathcal{A}(M)$ for every $n \in \nn$. Since $a_1/p$ is the smallest element in $M^\bullet$, it is an atom. By inequality \eqref{eq:semiadditivity of valuations}, for $n > 1$, the $p$-adic valuation of every element of the monoid
		\[
			M_{n-1} = \bigg\langle \frac{a_1}{p}, \dots, \frac{a_{n-1}}{p^{n-1}} \bigg\rangle
		\]
		is at least $-n+1$. As a consequence, $a_n/p^n \notin M_{n-1}$. From the fact that $a_n/p^n$ is the smallest element in $M \! \setminus \! M_{n-1}$, we deduce that it is an atom. Hence $S \subseteq \mathcal{A}(M)$ and, therefore, $\mathcal{A}(M) = S$. Because every generating set of $M$ contains $S$, and $a_n > p^n$ for $n \in \nn$, it follows that $M$ is not strongly bounded. Yet, the monoid $M$ is atomic because it is generated by atoms. Note also that $\mathcal{A}(M) = \infty$. For each prime $p$, we have found a Puiseux monoid over $p$ that is atomic but not strongly bounded; thus, there are infinitely many atomic Puiseux monoids failing to be strongly bounded. In fact, from the way we constructed the Puiseux monoid $M$ over $p$, we can infer that for each prime $p$, there are infinitely many atomic Puiseux monoids over $p$ that are not strongly bounded.
	\end{proof}
	
	We have just found a subfamily of atomic Puiseux monoids that are not strongly bounded. By contrast, it is natural to ask whether the family of strongly bounded Puiseux monoids comprises all Puiseux monoids containing no atoms. We postpone the answer to this question until we prove Proposition \ref{prop:a generating set of a BPM contains a bounded generating set}.
	
	Let us introduce some terminology for those monoids containing no atoms. An integral domain is called \emph{antimatter domain} if it contains no irreducible elements. Antimatter domains have been in-depth studied by Coykendall et al. \cite{CDM99}. However, no relevant investigation has been carried out concerning monoids containing no atoms.
	
	\begin{definition}
		Let $M$ be a monoid. If $\mathcal{A}(M)$ is empty, we say that $M$ is an \emph{antimatter} monoid.
	\end{definition}
	
	We should point out that, in general, the concepts of antimatter and atomic monoids are independent. Abelian groups are atomic and antimatter. The additive monoid $\nn_0$ is atomic, but it is not antimatter. Also, the additive monoid $\qq_{\ge 0}$  is antimatter; however, it is not atomic. Finally, the set of polynomials
	\[
		M = \{\,p(x) \in \qq[x] \ | \ p(0) \in \zz\,\}
	\]
	endowed with the standard multiplication of polynomials is not atomic; this is proved in \cite{sC14}. In addition, since every prime, seen as a constant polynomial, is an atom of $M$, one finds that $M$ is not antimatter. Having indicated the independence of antimatter and atomic monoids in a general setting, we should notice that, in the particular case of Puiseux monoids, a nontrivial atomic monoid automatically fails to be antimatter; more generally, this is actually true for every nontrivial reduced monoid.
	
	At this point we know there are infinitely many atomic Puiseux monoids failing to be strongly bounded. For the sake of completeness, we will also construct in Proposition~\ref{prop:PM antimatter not strongly bounded} an infinite subfamily of antimatter Puiseux monoids whose members fail to be strongly bounded.
	
	We know that every generating set of a Puiseux monoid $M$ contains $\mathcal{A}(M)$. In particular, if $M$ is atomic, then every generating set of $M$ contains a generating subset consisting of atoms, namely $\mathcal{A}(M)$. This suggests the question of whether every generating set of a bounded (resp., strongly bounded) Puiseux monoid contains a bounded (resp., strongly bounded) generating subset. As we show now, we can reduce any generating set of a bounded Puiseux monoid to a bounded generating subset.
	
	\begin{prop} \label{prop:a generating set of a BPM contains a bounded generating set}
		If $M$ is a bounded Puiseux monoid, then every generating set of $M$ contains a bounded generating subset.
	\end{prop}
	
	\begin{proof}
		Let $R$ be a set of generators of $M$. Take $B$ to be a bounded subset of rational numbers such that $M = \langle B \rangle$. For each $b \in B$ define
		\[
			S = \bigcup_{b \in B} S_b, \ \text{ where } \ S_b = \{r \in R \mid r \text{ divides } b \text{ in } M\}.
		\]
		Since $b$ is an upper bound of $S_b$ for each $b$, the fact that $B$ is bounded implies that $S$ is also bounded. So $S$ is a bounded subset of $R$. We verify now that $S$ is a generating set of $M$. It is enough to check that $M \subseteq \langle S \rangle$. Take an arbitrary $r \in R$. Since $B$ generates $M$, there exist $k \in \nn$ and $b_1, \dots, b_k \in B$ such that $r = b_1 + \dots + b_k$. Because $M$ is generated by $R$, for each $i \in \{1,\dots,k\}$ there exist $n_i \in \nn$ and $r_{i1}, \dots, r_{in_i} \in R$ such that $b_i = r_{i1} + \dots + r_{in_i}$. Consequently, we have
		\begin{equation} \label{eq:generators of bdd Puiseux monoids contains bdd generating sets}
			r = \sum_{i=1}^k b_i = \sum_{i=1}^k \sum_{j=1}^{n_i} r_{ij}.
		\end{equation}
		Notice that for every $i = \{1,\dots,k\}$ and $j \in \{1,\dots,n_i\}$, the element $r_{ij}$ divides $b_i$ in $M$. Thus, equality \eqref{eq:generators of bdd Puiseux monoids contains bdd generating sets} forces $r \in \langle S \rangle$ and, therefore, $M = \langle R \rangle \subseteq \langle S \rangle$. Hence $S$ is a bounded subset of $R$ generating $M$.
	\end{proof}
	
	Proposition \ref{prop:a generating set of a BPM contains a bounded generating set} naturally suggests the question of whether every generating set of a strongly bounded Puiseux monoid contains a strongly bounded generating subset. Unlike its parallel statement for boundedness, this desirable claim does not hold for strongly bounded Puiseux monoids.
	
	\begin{example}
		Let $p$ be an odd prime, and let us consider the following two sets of rational numbers:
		\[
			S = \bigg\{\frac{2}{p^{2^n}} \ \bigg{|} \ n \in \nn \bigg\} \ \text{ and } \ T = \bigg\{\frac{p^{2^n} - 1}{p^{2^{n+1}}}, \frac{p^{2^n} + 1}{p^{2^{n+1}}}  \ \bigg{|} \ n \in \nn \bigg\}.
		\]
		To verify that $S$ and $T$ generate the same Puiseux monoid, it suffices to notice that
		\[
			\frac{2}{p^{2^n}} = \frac{p^{2^n} - 1}{p^{2^{n+1}}} + \frac{p^{2^n} + 1}{p^{2^{n+1}}} \ \text{ and } \ \frac{p^{2^n} \pm 1}{p^{2^{n+1}}} = \frac{p^{2^n} \pm 1}{2} \frac{2}{p^{2^{n+1}}}.
		\]
		Let $M$ be the Puiseux monoid generated by any of the sets $S$ or $T$. Since $S$ is strongly bounded then so is $M$. On the other hand, every strongly bounded subset of $T$ must contain only finitely many elements; this is because the sequences of numerators of $\{(p^{2^n} -1)/p^{2^{n+1}}\}$ and $\{(p^{2^n} +1)/p^{2^{n+1}}\}$ both increase to infinite. In addition, as $M$ is antimatter, and so non-finitely generated, any subset of $T$ generating $M$ must contain infinitely many elements. Hence we can conclude that $T$ does not contain any strongly bounded subset generating $M$.
	\end{example}
	
	Let us resume now our search for a family of antimatter Puiseux monoids failing to be strongly bounded. To accomplish this goal, we make use of Proposition \ref{prop:a generating set of a BPM contains a bounded generating set}.
	
	\begin{prop} \label{prop:PM antimatter not strongly bounded}
		There exist infinitely many antimatter Puiseux monoids that are not strongly bounded.
	\end{prop}
	
	\begin{proof}
		Since every strongly bounded Puiseux monoid is also bounded, it is enough to find a family failing to be bounded. Let $p$ be an odd prime, and consider the sets
		\[
			S_p = \bigg\{ \frac{p^2 + 1}{p} + \frac{1}{2^n} \ \bigg{|} \ n \in \nn \bigg\} \ \text{ and } \ T = \bigg\{ \frac{1}{2^n} \ \bigg{|} \ n \in \nn \bigg\}.
		\]
		Let $\mathcal{P}$ be the collection of all infinite subsets of odd prime numbers. Let $P \in \mathcal{P}$ and take $M_{P}$ to be the Puiseux monoid generated by the set
		\[
			X = T \cup \big( \bigcup_{p \in P} S_p \big).
		\]
		We claim that $M_P$ is antimatter but not bounded. Let us verify first that $M_P$ is antimatter. Notice that for every $p \in P$ and $n \in \nn$,
		\[
			\frac{p^2 + 1}{p} + \frac{1}{2^n} =  \bigg( \frac{p^2 + 1}{p} + \frac{1}{2^{n+1}} \bigg) + \frac{1}{2^{n+1}},
		\]
		which means that $S_p \subseteq X + X$. Additionally, $1/2^n = 2(1/2^{n+1})$ and so $T \subseteq X + X$. Therefore $X \subseteq X + X$, which implies that $X$ does not contain any atoms of $M_P$. Since $M_P = \langle X \rangle$, it follows immediately that $M_P$ is antimatter.
		
		Now we show that $M_P$ is not bounded. Suppose, by way of contradiction, that $M_P$ is bounded. By Proposition \ref{prop:a generating set of a BPM contains a bounded generating set}, the set $X$ must contain a bounded subset $Y$ generating $M_P$. Observe that, for every prime $p \in P$, the set $S_p$ is bounded from below by $p$. Therefore there exists a natural $N$ such that $Y \cap S_p$ is empty for all prime $p > N$. If $q \in P$ is a prime greater than $N$, then $q$ does not divide $2q^2 + q + 2$ and so
		\[
			\frac{q^2 + 1}{q} + \frac{1}{2} = \frac{2q^2 + q + 2}{2q} \in S_q \setminus \langle Y \rangle,
		\]
		which contradicts the fact that $Y$ generates $M_P$. Hence, for each $P \in \mathcal{P}$, the Puiseux monoid $M_P$ is not bounded. Since $M_P \neq M_{P'}$ when $P$ and $P'$ are distinct elements of $\mathcal{P}$, one gets that $\{M_P \mid P \in \mathcal{P}\}$ is an infinite family of antimatter Puiseux monoids that are not bounded.
	\end{proof}
	
	As shown in the above example, not every generating set of a strongly bounded Puiseux monoid $M$ can be reduced to a strongly bounded subset generating $M$. However, if $M$ is not only strongly bounded but also atomic, then any set of generators of $M$ can certainly be reduced to a strongly bounded generating set. We record this observation in Proposition \ref{prop:a generating set of an atomic SBPM contains a SB generating set}, whose proof follows straightforwardly from the fact that the set of atoms of a reduced monoid must be contained in every generating set.
	
	\begin{prop} \label{prop:a generating set of an atomic SBPM contains a SB generating set}
		Let $M$ be a strongly bounded Puiseux monoid. If $M$ is atomic, then every generating set of $M$ contains a strongly bounded generating subset. \\
	\end{prop}
	
	\section{Atomic Structure of Strongly Bounded Puiseux Monoids} \label{sec:Atomic Structure of Strongly Bounded Puiseux Monoids}
	
	In this section, we restrict attention to the atomic structure of the subfamily of Puiseux monoids that happen to be strongly bounded. First, we find a condition under which members of this subfamily are antimatter; presenting Theorem \ref{theo:sufficiency for anatomicity} as the first main result. Then we move our focus to the classification of the atomic subfamily of strongly bounded Puiseux monoids over a finite set of primes $P$, which is stated in our second main result, Theorem \ref{theo: bounded P-RM are NS or non-atomic}.
	
	Let us introduce some definitions. We say that a sequence $\{a_n\}$ of integers \emph{stabilizes} at a positive integer $d$ if there exists $N \in \nn$ such that $d$ divides $a_n$ for every $n \ge N$. The \emph{spectrum} of a sequence $\{a_n\}$, denoted by $\text{Spec}(\{a_n\})$, is the set of primes $p$ for which $\{a_n\}$ stabilizes at $p$.
	
	\begin{lemma} \label{lem:empty spectrum}
		Let $\{a_n\}$ be a sequence of positive integers having an upper bound $B$. If the spectrum of $\{a_n\}$ is empty, then for each $N \in \nn$ there exist $k \le B+1$ and $n_1, \dots, n_k \in \nn$ such that $N < n_1< \dots < n_k$ and $\gcd(a_{n_1}, \dots, a_{n_k}) = 1$.
	\end{lemma}
	
	\begin{proof}
		Fix $N \in \nn$. Since the sequence $\{a_n\}$ is bounded, there are only finitely many primes dividing at least one of the terms of $\{a_n\}$. Let $P$ be the set comprising such primes. If $P$ is empty, then $a_n = 1$ for all $n \in \nn$, and we can take $k=2$ and both $n_1$ and $n_2$ to be two distinct integers greater than $N$ such that $N < n_1 < n_2$. In such a case, $k = 2 \le B+1$ for every upper bound $B$ of $\{a_n\}$, and one has $N < n_1 < n_2$ and $\gcd(a_{n_1}, a_{n_2}) = \gcd(1,1) = 1$. Assume, therefore, that $\{a_n\}$ is not the constant sequence whose terms are all ones. Thus, $P$ is not empty; let $P = \{p_1, \dots, p_k\}$. The fact that $2 \le p_n \le B$ when $n \in \{1, \dots, k\}$ implies $k \le B+1$. As the spectrum of $\{a_n\}$ is empty, there exists $n_1 > N$ such that $p_1$ does not divide $a_{n_1}$. Similarly, there exists $n_2 > n_1$ for which $p_2$ does not divide $a_{n_2}$. In general, if for $i < k$ one has chosen $n_1, \dots, n_i \in \nn$ so that $N < n_1 < \dots < n_i$ and $p_i \nmid a_{n_i}$, then there exists $n_{i+1} \in \nn$ satisfying $n_{i+1} > n_i$ and $p_{i+1} \nmid a_{n_{i+1}}$. After following the described procedure finitely many times, we will obtain $n_1, \dots, n_k \in \nn$ such that $N < n_1 < \dots < n_k$ and $p_i \nmid a_{n_i}$. Now it follows immediately that $\gcd(a_{n_1}, \dots, a_{n_k}) = 1$.
	\end{proof}
	
	The next theorem gives a sufficient condition for a Puiseux monoid to be antimatter.
	
	\begin{theorem} \label{theo:sufficiency for anatomicity}
		Let $\{ r_n \mid n \in \nn \}$ be a strongly bounded subset of rationals generating $M$. If $\mathsf{d}(r_n)$ divides $\mathsf{d}(r_{n+1})$, the sequence $\{\mathsf{d}(r_n)\}$ is unbounded, and the spectrum of $\{\mathsf{n}(r_n)\}$ is empty, then $M$ is antimatter.
	\end{theorem}
	
	\begin{proof}
		For every $n \in \nn$, let us denote $\mathsf{n}(r_n)$ and $\mathsf{d}(r_n)$ by $a_n$ and $b_n$, respectively. Let $B$ be an upper bound for the sequence $\{a_n\}$. Fix an arbitrary positive integer $N$. We will show that $b_N^{-1}$ is contained in $M$. By Lemma \ref{lem:empty spectrum}, there exist $k \in \nn$ and $n_1, \dots, n_k \in \nn$ such that $n_1 < \dots < n_k$ and $\gcd(a_{n_1}, \dots, a_{n_k}) = 1$, where $n_1$ is large enough to satisfy $B^2 < b_{n_1}b_N^{-1}$ (we are using here the unboundedness of $\{b_n\}$). On the other hand, $\gcd(a_{n_k}, b_{n_k}b_{n_i}^{-1}) = 1$; this makes sense because $b_{n_i}$ divides $b_{n_k}$ for $i = 1, \dots, k$. Since $\gcd(a_{n_1}, \dots, a_{n_k}) = 1$ and $\gcd(a_{n_k}, b_{n_k}b_{n_i}^{-1}) = 1$ for every $i \le k$, one has $\gcd(b_{n_k}b_{n_1}^{-1}a_{n_1}, \dots, b_{n_k}b_{n_{k-1}}^{-1}a_{n_{k-1}}, a_{n_k}) = 1$. Let $F(S)$ be the Frobenius number of the numerical semigroup $S = \langle b_{n_k}b_{n_1}^{-1}a_{n_1}, \dots, b_{n_k}b_{n_{k-1}}^{-1}a_{n_{k-1}}, a_{n_k} \rangle$. Taking $j \in \{1, \dots, k\}$ so that $b_{n_k}b_{n_j}^{-1}a_{n_j} = \max\{b_{n_k}b_{n_i}^{-1}a_{n_i} \mid 1 \le i \le k\}$ and using Theorem \ref{theo:Frobenius upper bound}, one can see that
		\begin{align*}
			F(S) < (a_{n_k} - 1)(b_{n_k}b_{n_j}^{-1}a_{n_j} - 1) < a_{n_k}a_{n_j}b_{n_k}b_{n_j}^{-1} \le B^2b_{n_k}b_{n_1}^{-1} < b_{n_k}b_N^{-1},
		\end{align*}
		where the last inequality follows from $B^2 < b_{n_1}b_N^{-1}$. As $b_{n_k}b_N^{-1} > F(S)$, there exist $c_1, \dots, c_k \in \nn_0$ such that
		\[
			b_{n_k}b_N^{-1} = \sum_{j=1}^k c_j b_{n_k}b_{n_j}^{-1}a_{n_j}
		\]
		and, accordingly,
		\[
			\frac{1}{b_N} = \sum_{j=1}^k c_j \frac{a_{n_j}}{b_{n_j}} \in M.
		\]
		Therefore $1/b_n \in M$ for every $n \in \nn$. Since $1/b_n = (b_{n+1}/b_n)1/b_{n+1}$ for every $n \in \nn$, none of the elements $1/b_n$ is an atom of $M$. Moreover, each generator $a_n/b_n$ can be written as the sum of $a_n$ copies of $1/b_n$; hence $a_n/b_n \notin \mathcal{A}(M)$ for all $n \in \nn$. Having checked that none of the generators of $M$ is an atom, we can conclude that $M$ is antimatter.
	\end{proof}
	
	There is some additional information about $M$ in the proof of Theorem \ref{theo:sufficiency for anatomicity}. We list it in the following corollary.
	
	\begin{cor} \label{cor:empty spectrum strongly bdd PM fully spans}
		If $M$ is a Puiseux monoid satisfying the conditions in Theorem \ref{theo:sufficiency for anatomicity}, then
		\[
			M = \bigg\{\frac{m}{b_n} \ \bigg| \ m \in \nn_0 \text{ and } n \in \nn \bigg\}.\\
		\]
	\end{cor}
	
	Let us verify that Theorem \ref{theo:sufficiency for anatomicity} is sharp, meaning that none of its hypotheses are redundant. First, if one drops the condition of $M$ being strongly bounded, then it might not be antimatter; see Proposition \ref{prop:PM atomic not strongly bounded}. Further, Example~ \ref{ex:atomic example with infinitely many atoms} illustrates that the condition $\mathsf{d}(r_n) \mid \mathsf{d}(r_{n+1})$ for every $n \in \nn$ is also required. Numerical semigroups are evidence that the sequence $\{\mathsf{d}(r_n)\}$ has to be necessarily unbounded. Finally, the family of strongly bounded Puiseux monoids constructed in Proposition \ref{prop:nontrivial family with finitely many atoms} shows that the emptiness of the spectrum of $\{\mathsf{n}(r_n)\}$ also needs to be imposed to guarantee $M$ is antimatter.

	\begin{prop} \label{prop:nontrivial family with finitely many atoms}
		For each $m \in \nn$ there exists a non-finitely generated strongly bounded Puiseux monoid having exactly $m$ atoms.
	\end{prop}
	
	\begin{proof}
		Take $p$ and $q$ to be prime numbers satisfying $p \neq q$ and $q > m$. Consider the Puiseux monoid over $p$
		\[
			M = \bigg\langle m, \dots, 2m-1, \frac{q}{p^{m+1}}, \frac{q}{p^{m+2}}, \dots \bigg\rangle.
		\]
		We check that $\mathcal{A}(M) = \{m, \dots, 2m-1\}$. Suppose $a \in \zz$ such that $m \le a \le 2m-1$. Since $a \in M$, it can be written as
		\[
			a = a' + \sum_{n \ge 1} c_n \frac{q}{p^{m+n}},
		\]
		for $a' \in \{0\} \cup \{m, \dots, 2m-1\}$ and for a suitable set of non-negative coefficients $c_n$, all but finitely many of them being zero. Then
		\begin{equation}
			\frac{(a - a')p^m}{q} = \sum_{n \ge 1} \frac{c_n}{p^n}. \label{eq:example with finitely many atoms}
		\end{equation}
		Because the $q$-adic valuations of the right-hand side of \eqref{eq:example with finitely many atoms} are at least zero,  the left-hand side of this equation must be an integer. As a result, $a - a' = 0$ and so $c_n = 0$ for every $n \in \nn$. This implies that $a \in \mathcal{A}(M)$. Thus, $\{m, \dots, 2m-1\} \subseteq \mathcal{A}(M)$. On the other hand, no generator of the form $q/p^n$ for $n > m$ can be an atom of $M$, for $q/p^n$ is the sum of $p$ copies of $q/p^{n+1}$. Hence $\mathcal{A}(M) = \{m, \dots, 2m-1\}$.
	\end{proof}
	
	Proposition \ref{prop:nontrivial family with finitely many atoms} tells us that there are infinitely many Puiseux monoids (by varying our choice of $p$) with any fixed finite number of atoms that are not finitely generated and, therefore, non-atomic. This fact, along with Example \ref{ex:nonfinite generated with finite number of atoms trivial} and Example \ref{ex:nonfinite generated with infinitely many atoms}, gives evidence of the complexity of the atomic structure of Puiseux monoids. \\
	
	We conclude our discussion about the atomic structure of Puiseux monoids with a classification of the atomic subfamily of strongly bounded Puiseux monoids over a finite set of primes. First, let us introduce some terminology. The \emph{spectrum} of a natural $n$, which is denoted by $\text{Spec}(n)$, is the set of all prime divisors of $n$. In addition, given a finite set of primes $P = \{p_1, \dots, p_k\}$ with $p_1 < \dots < p_k$, the \emph{support} of $n \in \nn$ with respect to $P$, denoted by $\text{Supp}_P(n)$, is the set of indices $i$ such that $p_i \mid n$. The next lemmas will be used in the proof of Theorem~\ref{theo: bounded P-RM are NS or non-atomic}.
	
	\begin{lemma} \label{lem:simultaneously increasing subsequence implies antimatter}
		Let $P$ be a finite set of primes, and let $M$ be a Puiseux monoid over $P$. If there is a sequence $\{b_n\} \subseteq \nn$ such that $1/b_n \in M$ for every $n \in \nn$ and $\{\pval(b_n)\}$ is strictly increasing for each $p \in P$, then $M$ is antimatter.
	\end{lemma}
	
	\begin{proof}
		Fix a sequence $\{b_n\}$ of positive integers such that $1/b_n \in M$ and $\{{\pval}(b_n)\}$ is strictly increasing for each $p \in P$. Take an element $q$ in $M^\bullet$, and let $a,b \in \nn$ such that $q = a/b$ and $\gcd(a,b) = 1$. Set
		\[
			m_q = \min_{p \in P}\{\pval(q)\}.
		\]
		Since $\{{\pval(b_n)}\}$ is strictly increasing for each $p \in P$ and the set $P$ is finite, there exists $N \in \nn$ such that $-\pval(b_n) < m_q$ for every $p \in P$ and $n \in \zz_{> N}$. Therefore we obtain
		\begin{equation} \label{eq:sum of copies of elements in the increasing-exponent sequence}
			q = \bigg(d\prod_{p \in P}p^{\pval(q) - \pval(1/b_n)}\bigg)\bigg(\prod_{p \in P}p^{\pval(1/b_n)}\bigg) 
			= \bigg(d\prod_{p \in P}p^{\pval(q) - \pval(1/b_n)}\bigg) \frac{1}{b_n}, 
		\end{equation}
		where $d$ is the greatest natural number dividing $a$ such that $p \nmid d$ for every $p \in P$. Because $\pval(1/b_n) = - \pval(b_n) < m_q \le \pval(q)$ for each $p \in P$, it follows that the exponents $\pval(q) - \pval(1/b_n)$ in \eqref{eq:sum of copies of elements in the increasing-exponent sequence} are all positive. Thus, $q$ is the sum of more than one copy of $1/b_n$, whence we find that $q$ is not an atom. Since $q$ was taken arbitrarily in $M^\bullet$, we conclude that $\mathcal{A}(M)$ is empty.
	\end{proof}
	
	\begin{lemma} \label{lem:common bound for simultaneous strictly increasing valuations}
		Let $k \in \nn$ and $P = \{p_1, \dots,p_k\}$ be a finite set of primes. Let $\{s_n\}$ be a sequence so that $\emph{Spec}(s_n) \subseteq P$ for every $n \in \nn$. Then there exists $N \in \nn$ satisfying the following property: if there exist $n \in \nn$ and $I \subseteq \{1, \dots,k\}$ such that ${\pval}_i(s_n) > N$ for each $i \in I$, then there is a subsequence $\{s'_n\}$ of $\{s_n\}$ for which $\{{\pval}_i(s'_n)\}$ is strictly increasing for each $i \in I$.
	\end{lemma}
	
	\begin{proof}
		Let $\mathcal{J}$ be the set of all subsets of indices $J$ of $\{1, \dots,k\}$ for which $\{s_n\}$ does not contain any subsequence $\{s'_n\}$ such that $\{{\pval}_j(s'_n)\}$ is strictly increasing for each $j \in J$. For each $J \in \mathcal{J}$ there must exist $N_J \in \nn$ satisfying that, for every $n \in \nn$, the inequality ${\pval}_j(s_n) \le N_J$ holds for at least an index $j \in J$. Take $N \in \nn$ to be $\max\{N_J \mid J \subseteq \{1, \dots,k\}\}$. Suppose now that $n$ is a natural number and $I$ is a subset of $\{1, \dots, k\}$ such that ${\pval}_i(s_n) > N$ for each $i \in I$. If $I \notin \mathcal{J}$, then we are done. Suppose, by way of contradiction, that $I \in \mathcal{J}$. Then $N \ge N_I$ and so ${\pval}_i(s_n) > N_I$ for every $i \in I$. This means that the inequality ${\pval}_i(s_n) \le N_I$ does not hold for any $i \in I$, contradicting the fact that $I \in \mathcal{J}$.
	\end{proof}
	
	\begin{lemma} \label{lem:partitioning the set of generators}
		Let $M$ be a Puiseux monoid generated by a set $S$. Suppose also that $S = S_1 \cup \dots \cup S_n$, where $n \in \nn$ and $S_i$ is a nonempty subset of $S$ for $i=1, \dots,n$. Then the next set inclusion holds:
		\begin{equation}
			\mathcal{A}(M) \subseteq \bigcup_{i=1}^n \mathcal{A}(\langle S_i \rangle). \label{eq:distribution of atoms in finitely many submonoids}
		\end{equation}
	\end{lemma}
	
	\begin{proof}
		If $\mathcal{A}(M)$ is empty, then \eqref{eq:distribution of atoms in finitely many submonoids} follows trivially. So assume $\mathcal{A}(M)$ is not empty, and take $a$ to be an atom of $M$. Then $a \in S$, and therefore there exists $i \in \{1,\dots,n\}$ such that $a \in S_i \subseteq \langle S_i \rangle$. Because $\mathcal{A}(M) \cap \langle S_i \rangle \subseteq \mathcal{A}(\langle S_i \rangle)$, one gets
		\[
			a \in \mathcal{A}(\langle S_i \rangle) \subseteq \bigcup_{i=1}^n \mathcal{A}(\langle S_i \rangle).
		\]
		Since $a$ was taken arbitrarily in $\mathcal{A}(M)$, the inclusion \eqref{eq:distribution of atoms in finitely many submonoids} holds, as desired.
	\end{proof}
	
	We are now in a position to prove our last main result.
	
	\begin{theorem} \label{theo: bounded P-RM are NS or non-atomic} 
		Let $M$ be a strongly bounded finite Puiseux monoid. Then $M$ is atomic if and only if $M$ is isomorphic to a numerical semigroup.
	\end{theorem}
	
	\begin{proof}
		Let $P$ be a set of primes such that $M$ is finite over $P$. First, we will prove that $M$ has only finitely many atoms. We proceed by induction on the cardinality of $P$. If $|P|=0$, then $M$ is isomorphic to a numerical semigroup and, therefore, $\mathcal{A}(M)$ is finite. Suppose $k$ is a positive integer such that the statement of the theorem is true when $|P| < k$. We shall prove that every strongly bounded Puiseux monoid $M$ over $P$ has finitely many atoms when $|P|=k$. Set $P = \{p_1, \dots, p_k\}$. In addition, let $\{a_n\}$ and $\{b_n\}$ be two sequences of natural numbers such that $\{a_n\}$ is bounded and $\gcd(a_n,b_n)=1$ for every $n \in \nn$. Let
		\[
			M = \langle S \rangle, \text{ where } S = \bigg\{\frac{a_n}{b_n} \ \bigg{|} \ n \in \nn \bigg\}.
		\]
		To show that $M$ has only finitely many atoms, we will distribute the generators $a_n/b_n$ of $M$ into finitely many submonoids of $M$ and then we will apply Lemma~\ref{lem:partitioning the set of generators}. Since the sequence $\{a_n\}$ is bounded, it has a maximum, namely $m$. Set
		\[
			S_{j,I} = \bigg\{\frac{a_n}{b_n} \ \bigg{|} \ n \in \nn, \ a_n = j, \text{ and }\text{Supp}_P(b_n) = I \bigg\}
		\]
		for each $j \in \{1, \dots,m\}$ and $I \subseteq \{1, \dots,k\}$. Let $\mathcal{J}$ be the set of pairs $(j,I)$ such that $S_{j,I}$ is not empty. Using Lemma \ref{lem:partitioning the set of generators}, we obtain
		\begin{equation}
			\mathcal{A}(M) \subseteq S \subseteq \bigcup_{(j,I) \in \mathcal{J}} \mathcal{A}(\langle S_{j,I} \rangle). \label{eq:distribution of atoms}
		\end{equation}
		For each $(j,I) \in \mathcal{J}$, let $M_{j,I}$ be the Puiseux monoid over $P$ generated by $S_{j,I}$. By the inclusion \eqref{eq:distribution of atoms}, we are done once we show that $M_{j,I}$ contains only finitely many atoms for every pair $(j,I) \in \mathcal{J}$.
		
		We fix an arbitrary pair $(j,I) \in \mathcal{J}$ and prove that $\mathcal{A}(M_{j,I})$ is finite. Since $M_{j,I}$ is a strongly bounded finite Puiseux monoid over $P$, if $I$ is strictly contained in $\{1, \dots, k\}$, then $\mathcal{A}(M_{j,I})$ is finite (induction hypothesis). So it just remains to check that $M_{j,I}$ contains only finitely many atoms when $I = \{1, \dots,k\}$. If $S_{j,I}$ is finite, then by Proposition \ref{prop:finitely generated positive rational monoids} the monoid $M_{j,I}$ is isomorphic to a numerical semigroup; in this case, $\mathcal{A}(M_{j,I})$ is finite. So we will assume $S_{j,I}$ is not finite. Let $\{s_n\}$ be a subsequence of the sequence $\{b_n\}$ such that $S_{j,I} = \{j/s_n \mid n \in \nn\}$. If $\{s_n\}$ contains a subsequence $\{s'_n\}$ such that ${\pval}_i(s'_n)$ is strictly increasing for each $i \in \{1, \dots,k\}$, then Lemma \ref{lem:simultaneously increasing subsequence implies antimatter} implies that $M_{j,I}$ is antimatter and, therefore, contains  no atoms (notice that $M_{j,I}$ is isomorphic to $\langle 1/s_n \mid n \in \nn \rangle$ via $x \mapsto j^{-1}x$). So we assume such a subsequence of $\{s_n\}$ does not exist. Using Lemma \ref{lem:common bound for simultaneous strictly increasing valuations}, we can find $N \in \nn$ so that if for $I' \subseteq \{1, \dots,k\}$ there is $n \in \nn$ satisfying ${\pval}_i(s_n) > N$ for each $i \in I'$, then there exists a subsequence $\{s'_n\}$ of $\{s_n\}$ such that $\{{\pval}_i(s'_n)\}$ is strictly increasing for each $i \in I'$ (note that now $I'$ must be a proper subset of $\{1, \dots,k\}$). Set $\mu = p_1^N \dots p_k^N$ and $M' = \mu M_{j,I}$. Thus, one has that $M_{j,I}$ is isomorphic to $M'$ via multiplication by $\mu$, meaning $x \mapsto \mu x$. So it suffices to check that $M'$ has finitely many atoms.
		
		To prove that $M'$ has only finitely many atoms, consider its proper generating set $S' = \{a'_n/b'_n \mid n \in \nn\}$, where $a'_n/b'_n$ results from reducing the fraction $\mu j/s_n$ to lowest terms. For every $n \in \nn$ and $I \subseteq \{1, \dots,k\}$, it follows that $p_i$ divides $b'_n$ for each $i \in I$ if and only if ${\pval}_i(s_n) > N$ for each $i \in I$, which implies the existence of a subsequence $\{s'_n\}$ of $\{s_n\}$ such that $\{{\pval}_i(s'_n)\}$ is strictly increasing for each $i \in I$ (by Lemma \ref{lem:common bound for simultaneous strictly increasing valuations}). Since $\{s_n\}$ contains no subsequence $\{s'_n\}$ such that $\{{\pval}_i(s'_n)\}$ is strictly increasing for each $i \in \{1, \dots,k\}$, for every $n \in \nn$ there is at least an index $i \in \{1, \dots, k\}$ such that ${\pval}_i(s_n) \le N$, i.e., $\text{Supp}_P(b'_n)$ is a proper subset of $\{1, \dots, k\}$ for every $n \in \nn$. Set
		\[
			S_I = \bigg\{\frac{a'_n}{b'_n} \ \bigg{|} \ n \in \nn \ \text{ and } \ \text{Supp}_P(b'_n) = I \bigg\}
		\]
		for each $I \subset \{1, \dots,k\}$. Let $\mathcal{I}$ be the collection of all subsets $I$ of $\{1,\dots, k\}$ such that $S_I$ is not empty. Notice that every element of $\mathcal{I}$ is a proper subset of $\{1,\dots,k\}$. By Lemma \ref{lem:partitioning the set of generators},
		\begin{equation}
			\mathcal{A}(M') \subseteq S' \subseteq \bigcup_{I \in \mathcal{I}} \mathcal{A}(\langle S_I \rangle). \label{eq:distribution of atoms 2}
		\end{equation}
		Since each set of indices $I \in \mathcal{I}$ is strictly contained in $P$, by the induction hypothesis we get that $\langle S_I \rangle$ contains only finitely many atoms for every $I \in \mathcal{I}$. Therefore $|\mathcal{A}(M')| < \infty$ follows from \eqref{eq:distribution of atoms 2}.
		
		At this point, we have proved that every strongly bounded Puiseux monoid over a finite set of primes $P$ has finitely many atoms. Suppose $M$ is atomic. By Theorem \ref{theo:finitely generated rational monoids}, $M$ has a minimal set of generators, which must be $\mathcal{A}(M)$. Hence $M$ is finitely generated and, by Proposition \ref{prop:finitely generated positive rational monoids}, it is isomorphic to a numerical semigroup. Conversely, suppose $M$ is isomorphic to a numerical semigroup. Since every numerical semigroup is atomic, $M$ must be atomic. This completes the proof.
	\end{proof}
	
	Example \ref{ex:atomic example with infinitely many atoms} can be used as evidence that Theorem \ref{theo: bounded P-RM are NS or non-atomic} does not hold if we do not require $P$ to be finite. Besides, the strongly boundedness of the Puiseux monoid $M$ over $P$ is not superfluous, as one can see in Proposition \ref{prop:PM atomic not strongly bounded}, which guarantees the existence of an atomic Puiseux monoid over a prime $p$ with infinitely many atoms that fails to be strongly bounded. \\
	
	\section{Acknowledgments}
	
	While working on this paper, the author was supported by the UC Berkeley Chancellor Fellowship. The author is grateful to Scott Chapman for valuable feedback and encouragement, to Marly Gotti for her dedicated final review, and to Chris O'Neill for helpful discussions on early drafts of this paper. Also, the author thanks an anonymous referee for suggesting recommendations that improved this paper. \\

\end{document}